\documentclass[12pt, twoside, reqno]{amsart}
\usepackage{amscd}
\usepackage{bbm}
\usepackage{dsfont}
\usepackage{enumerate}
\usepackage{latexsym}
\usepackage{srcltx}
\usepackage{amsfonts}
\usepackage{amssymb}
\usepackage{datetime}
\usepackage{setspace}
\usepackage{enumerate}
\usepackage{amsthm}
\usepackage{graphicx}
\usepackage{epstopdf}

\newtheorem{theorem}{Theorem}[section]
\newtheorem{proposition}[theorem]{Proposition}
\newtheorem{lemma}[theorem]{Lemma}
\newtheorem{corollary}[theorem]{Corollary}

\theoremstyle{definition}
\newtheorem{example}[theorem]{Example}
\newtheorem{definition}[theorem]{Definition}

\newtheorem{remark} [theorem] {Remark}

\newtheorem{problem} [theorem] {Problem}

\begin{document}
\title{ Spectrum of Weighted Composition Operators \\
Part IX \\
The spectrum and essential spectra of some weighted composition operators
on uniform algebras.}

\author{Arkady Kitover}

\address{Community College of Philadelphia, 1700 Spring Garden St., Philadelphia, PA, USA}

\email{akitover@ccp.edu}

\author{Mehmet Orhon}

\address{University of New Hampshire, 105 Main Street
Durham, NH 03824}

\email{mo@unh.edu}


\date{\today}

\begin{abstract}
  We obtain some results about the spectrum and the upper semi-Fredholm spectrum of weighted composition operators on uniform algebras, assuming that the corresponding map maps the Shilov boundary onto itself. In particular, it follows from our results that in the case of analytic uniform algebras the spectrum is a connected rotation invariant subset of the complex plane, and that the upper semi-Fredholm spectrum is rotation invariant as well.
\end{abstract}

\maketitle

\markboth{Arkady Kitover and Mehmet Orhon}{Spectrum of weighted composition operators. IX}

\section{Introduction and preliminaries}

This paper represents a continuation of the investigation of the spectrum and  essential spectra of weighted composition operators acting on uniform algebras started in~\cite{Ki} and continued in~\cite{KO}. In~\cite{Ki} the first named author obtained a description of the spectrum and approximate point spectrum of weighted automorphisms of unital uniform algebras, while in~\cite{KO} we obtained some additional information about essential spectra of weighted automorphisms of unital uniform algebras for some special types of automorphisms (e.g. automorphisms of the disk-algebra induced by elliptic Möbius transformations).

The goal of this paper is to obtain some information about the spectra of weighted isometric endomorphisms of uniform algebras. Our main results are located in Section 2. The first of them (Theorem~\ref{t1}) states that the upper semi-Fredholm spectrum of such a weighted endomorphism coincides with its approximative point spectrum, the second one (Theorem~\ref{t2}) provides a criterion for the presence (or absence) of circular holes in the spectrum, while Theorem\ref{t3} and Proposition~\ref{p2}  provide some information about the semi-Fredholm spectra in some special cases. 

As a corollary of our main results (Corollary~\ref{c3}) we show that the spectrum of a weighted isometric endomorphism of an analytic uniform algebra (see Definition~\ref{d1}) is a rotation invariant and connected subset of the complex plane.

In section 3 we apply our main results to obtain a description of essential spectra of weighted automorphisms of the disc-algebra.
In section 4 we consider a few examples of applications of our main results.

\bigskip

\noindent Throughout the paper the following notations are used.

\noindent $A$ is a unital uniform algebra.

\noindent $\mathfrak{M}_A$ is the space of maximal ideals of $A$.

\noindent $\partial A$ is the Shilov boundary of $A$.

\noindent $C(K)$ is the space of all complex-valued continuous functions on a Hausdorff compact space $K$ endowed with the standard norm.

\noindent $\mathds{N}$ is the set of all positive integers.

\noindent $\mathds{C}$ is the field of all complex numbers.

\noindent All linear spaces are considered over the field $\mathds{C}$.

\noindent Let $K$ be a compact Hausdorff space and $\varphi$ be a continuous surjection of $K$ onto itself. Then, $\varphi^0(k) = k, k \in K$, and $\varphi^n = \varphi^{n-1} \circ \varphi, n \in \mathds{N}$.

\noindent If $E \subseteq K$ and $n \in \mathds{N}$ then $\varphi^{(-n)}(E) = \{k \in K : \varphi^n(k) \in E\}$ is the full $n^{th}$ $\varphi$-preimage of $E$.

\noindent Let $K$ be a Hausdorff compact space, $\varphi$ be a continuous map of $K$ into itself, and $w \in C(K)$. Then $w_1 = w$ and $w_n = w_{n-1}(w \circ \varphi^{n-1}), n = 2,3, \ldots$.

\noindent Let $X$ be a Banach space and $T : X \rightarrow X$ be a continuous linear operator. Then $\sigma(T,X)$ is the spectrum of $T$. 

\noindent $\sigma_{a.p.}(T,X)$ is the approximate point spectrum of $T$, i.e.
\begin{equation*}
   \sigma_{a.p.}(T,X) = \{\lambda \in \mathds{C} : \exists x_n \in X, \|x_n\|=1, Tx_n - \lambda x_n \mathop \rightarrow_{n \rightarrow \infty} 0\}.
\end{equation*}
According to this definition the point spectrum of $T$ is a subset of $\sigma_{a.p.}(T,X)$.

\noindent $\sigma_r(T,X) = \sigma(T,X) \setminus \sigma_{a.p.}(T,X)$ is the residual spectrum of $T$.

\noindent Recall that a continuous linear operator $T$ on a Banach space $X$
 is called upper semi-Fredholm (respectively, lower semi-Fredholm) if its image $R(T)$ is closed in  $X$ and $\dim \ker T < \infty$ (respectively, if $codim R(T) < \infty$).  

\noindent The operator $T$ is called semi-Fredholm if it is either upper semi-Fredholm or lower semi-Fredholm, and it is called Fredholm if it is both upper semi-Fredholm and lower semi-Fredholm.

\noindent The operator $T$ is called Weil operator if $\dim \ker T = codim R(T) < \infty$.

\bigskip

\noindent $\sigma_{sf}(T)$ is the semi-Fredholm spectrum of $T$. $\sigma_{sf}(T) = \{\lambda \in \mathds{C}: \lambda I - T$ is not semi-Fredholm$\}$.

\noindent $\sigma_{usf}(T,X)$ is the upper semi-Fredholm spectrum of $T$, i.e. $\sigma_{usf}(T) = \{\lambda \in \mathds{C}: \lambda I - T$ is not an upper semi-Fredholm operator$\}$. It is well known (see e.g.~\cite{EE}) that
\begin{equation*}
  \begin{split}
  & \sigma_{usf}(T,X) = \{\lambda \in \mathds{C} : \exists x_n \in X, \|x_n\|=1, Tx_n - \lambda x_n \mathop \rightarrow_{n \rightarrow \infty} 0 ,\\
  & \text{and the sequence} \; x_n \; \text{is singular, i.e. it does not contain} \\
   & \text{a norm convergent subsequence}\}.
  \end{split}
\end{equation*}

\noindent $\sigma_{lsf}(T,X)$ is the lower semi-Fredholm spectrum of $T$, i.e. $\sigma_{lsf}(T) = \{\lambda \in \mathds{C}: \lambda I - T$ is not a lower semi-Fredholm operator$\}$. 
\begin{equation*}
  \begin{split}
  & \sigma_{lsf}(T,X) = \{\lambda \in \mathds{C} : \exists x_n^\prime \in X^\prime, \|x_n^\prime\|=1, T^\prime x_n^\prime - \lambda x_n^\prime \mathop \rightarrow_{n \rightarrow \infty} 0 ,\\
  & \text{and the sequence} \; x_n^\prime \; \text{is singular}.
    \end{split}
\end{equation*}

\noindent $\sigma_f(T)$ is the Fredholm spectrum of $T$, i.e. $\sigma_f(T) = \{\lambda \in \mathds{C}: \lambda I - T$ is not a Fredholm operator$\}$.  

\noindent $\sigma_w(T)$ is the Weil spectrum of $T$, i.e. $\sigma_f(T) = \{\lambda \in \mathds{C}: \lambda I - T$ is not a Weil operator$\}$. 

\noindent It is well known (see~\cite{EE}) that
\begin{itemize}
  \item $\sigma_{sf}(T) \subseteq \sigma_{usf}(T) \subseteq \sigma_f(T) \subseteq \sigma_w(T)$,
  \item $\sigma_{usf}(T) = \sigma_{lsf}(T^\prime)$, $\sigma_{lsf}(T) = \sigma_{usf}(T^\prime)$.
\end{itemize}

\noindent Let $A$ be a uniform algebra and $U$ be an isometric endomorphism of $A$. Let $\varphi : \mathfrak{M}_A \rightarrow \mathfrak{M}_A$ be the corresponding continuous map of the space of maximal ideals of $A$ into itself. Then, $\varphi(\partial A) = \partial A$. We agree to denote the endomorphism $U$ as $T_\varphi$.

Recall that a point $t \in \partial A$ is called a $p$-point if $\{t\} = \bigcap \limits_\alpha E_\alpha$, where the sets $E_\alpha \subset \partial A$ are peak sets, i.e. there are $f_{\alpha} \in A$ such that $f_\alpha \equiv 1$ on $E_\alpha$ and $|f(s)| < 1, s \in \partial A \setminus E_\alpha$. The set of $p$-points is dense in $\partial A$ (see e.g.~\cite[p.12]{Gr}). The following lemma is probably known.

\begin{lemma} \label{l1}
  Let $A$ be a unital uniform algebra and $t \in \partial A$ is a $p$-point. Then the characteristic function $\chi_{\{t\}}$ of the singleton $\{t\}$ can be identified with an element of the second dual $A^{\prime \prime}$.
\end{lemma}

\begin{proof}
  We need to prove that if $\mu \in C(\partial A)^\prime$ and $\mu \perp A$ then $\mu(\{t\})=0$. Let $ E \subset \partial A$ be a peak set and let $f \in A$ be a function that peaks on $E$. Then $\int{f^n d\mu}=0, n \in \mathds{N}$, and therefore $\int{ \chi_E d\mu} =0$. Because the set $\{t\}$ is intersection of peak sets and $\mu$ defines an order continuous functional on $C(\partial A)^{\prime \prime}$, we have $\int{\chi_{\{t\}}}d\mu = 0$. Thus, $\chi_{\{t\}} \in A^{\perp \perp} = A^{\prime \prime}$. 
  \end{proof}

\section{The main results}

\begin{theorem} \label{t1}
  Let $A$ be a unital uniform algebra, $T_\varphi$ is an isometric endomorphism of $A$ and $w \in A$. Let $T = wT_\varphi$. Assume that $\varphi$ is an open map of $\partial A$ onto itself. Then,
  
  \begin{equation*}
    \sigma_{usf}(T,A) = \sigma_{a.p.}(T,A) = \sigma_{a.p.}(T, C(\partial A)).
  \end{equation*}
\end{theorem}

\begin{proof}
  The inclusions $\sigma_{usf}(T,A) \subseteq \sigma_{a.p.}(T,A) \subseteq \sigma_{a.p.}(T, C(\partial A))$ are trivial. We need to prove the implication $\lambda \in \sigma_{a.p.}(T,C(\partial A)) \Rightarrow \lambda \in \sigma_{usf}(T,A)$.
  
  Assume first then $\lambda =0$. Let $f_n \in C(\partial_A)$ such that  $\|f_n\|=1$ and $Tf_n \rightarrow 0$.
Because $\partial_A$ has no isolated points there are open non-empty subsets $V_n$ of $\partial_A$ such that
\begin{equation}\label{eq1}
  cl V_n \cap cl V_m = \emptyset , m \neq n,  \; \text{and} \; \min_{cl V_n} |f_n| \geq 1 - 1/n.
\end{equation}
Let $g_n \in A$ be such that
\begin{equation}\label{eq2}
 \|g_n\| = 1 \; \text{and} \; \max_{\partial_A \setminus V_n} |g_n| \leq 1/n.
\end{equation}
It follows from~(\ref{eq1}) and ~(\ref{eq2}) that
\begin{equation}\label{eq3}
|g_n| \leq 2|f_n| + 1/n\; \text{and}\;  g_mg_n \mathop{\rightarrow}_{n> m \to \infty} 0.
\end{equation}
It follows from~(\ref{eq3}) that $Tg_n \rightarrow 0$ and the sequence $g_n$ is singular. Thus, $0 \in \sigma_{usf}(T,A)$.

Next, we assume that $\lambda \in \sigma_{a.p.}(T, C(\partial_A))$ and $\lambda \neq 0$. Without loss of generality we can assume that $\lambda = 1$. 

We will prove that $1 \in \sigma_{a.p.}(T,A)$. Let $f_n \in C(\partial A)$ be such that $\|f_n\| = 1$ and $Tf_n - f_n \rightarrow 0$. We have to consider two cases.

(a) For any $N \in \mathds{N}$ there are an $n > N$ and $t_n \in \partial A$ such that $|f_n(t_n)| > 1 - 1/n$ and the points $\varphi^i(t_n), i = 1, \ldots N$ are pairwise distinct. Because the map $\varphi : \partial A \rightarrow \partial A$ is open and the set of  $p$-points is dense in $\partial A$,  for any $ N \in \mathds{N}$ we can find $n = n(N) \in \mathds{N}$ and $s_n \in \partial A$ such that the point $u_n = \varphi^N(s_n)$ is a $p$-point, the points $\varphi^i(s_n), i = 1, \ldots N$ are pairwise distinct, and 

\begin{equation}\label{eq12}
    \begin{split}
    & |f_n(s_n)| > 1 - 1/n, \\
    & \|T^if_n - f_n\| \leq 1/N, i = 1, \ldots, 2N+1, \\
    & |w_N(s_n)| > 1/2.
  \end{split}
\end{equation}
Let us fix an $N$ and a corresponding $n = n(N)$ such that~(\ref{eq12}) holds .
By Lemma~\ref{l1} we can consider the function $\chi_n = \chi_{\{u_n\}}$ as an element of the second conjugate space $A^{\prime \prime}$. Let $g_N = \frac{1}{w_N(s_n)}\chi_n$ and 
\begin{equation}\label{eq13}
  h_N = \sum \limits_{i=0}^{2N} \bigg{(} 1 -\frac{1}{\sqrt{N}} \bigg{)}^{|i-N|} (T^{\prime \prime})^i g_N
\end{equation}
Let us notice that $h_n(s_n) = (T^{\prime \prime})^Ng_N(s_n) = w_N(s_n)g_N(u_n)=1$ and therefore $\|h_N\| \geq 1$. Moreover, it follows from~(\ref{eq12}) that 
\begin{equation}\label{eq15}
  |g_N| \leq 3|f_n \chi_n|.
\end{equation}

It follows from~(\ref{eq13}) that 

\begin{equation}\label{eq14}
\begin{split}
 & T^{\prime \prime}h_N - h_N = -\bigg{(}(1-\frac{1}{\sqrt{N}} \bigg{)}^N g_N \\
& - \frac{1}{\sqrt{N}} \sum \limits_{i=1}^n \bigg{(} 1 -\frac{1}{\sqrt{N}} \bigg{)}^{|i-N|} (T^{\prime \prime})^i g_N \\
& + \frac{1}{\sqrt{N}} \sum \limits_{i=N+1}^{2N} \bigg{(} 1 -\frac{1}{\sqrt{N}} \bigg{)}^{|i-N|} (T^{\prime \prime})^i g_N \\
& + \bigg{(}(1-\frac{1}{\sqrt{N}} \bigg{)}^N (T^{\prime \prime})^{2N+1} g_N.
 \end{split}
\end{equation}
Notice also that it follows from~(\ref{eq15}) and~(\ref{eq12})  that
\begin{equation}\label{eq16}
  \|(T^{\prime \prime})^{2N+1}g_N\| \leq 3\|(T^{\prime \prime})^{2N+1}f_n\| \leq 6.
\end{equation}

Combining~(\ref{eq12}), ~(\ref{eq15}), ~(\ref{eq14}), and~(\ref{eq16})  we obtain that
\begin{equation}\label{eq17}
  \|T^{\prime \prime}h_N - h_N\| \leq \frac{1}{\sqrt{N}} \|h_N\| +9\bigg{(}(1-\frac{1}{\sqrt{N}} \bigg{)}^N.
\end{equation}
Because $N$ is arbitrary large it follows from~(\ref{eq17}) that $1 \in \sigma_{a.p.}(T^{\prime \prime}, A^{\prime \prime}) = \sigma_{a.p.}(T,A)$. 

(b) If the assumptions we made in case (a) are not satisfied, then without loss of generality we can assume that there are $f_n \in C(\partial A)$, open subsets $V_n \subset \partial A$, a positive integer $m$, and a $c>0$ such that
\begin{enumerate} [(i)]
  \item $\|f_n\| = 1$ and $\|T^jf_n - f_n \| \leq 1/n, j = 1, \ldots, n$.
  \item $\varphi(V_n) = V_n, n \in \mathds{N}$, and all the points of the set $V_n$ are $\varphi$-periodic of the smallest period $m$.
  \item $\max \limits_{cl V_n} |f_n| > c, n \in \mathds{N}$.
  \end{enumerate}
Let $s_n \in V_n$ be a $p$-point such that $|f_n(s_n)| > c$ and let $\chi_n = \chi_{\{s_n\}} \in A^{\prime \prime}$. We claim that $\chi_{n,j} = \chi_{\{\varphi^j(s_n)\}}, j=1, \ldots m-1, \in A^{\prime \prime}$. Indeed, let $\mu \in A^\perp$, then by Lemma~\ref{l1}
\begin{equation*}
  \int \chi_{n,j} d\mu = \frac{1}{w_j(s_n)} \int (T^{\prime \prime})^j \chi_n d\mu = \frac{1}{w_j(s_n)} \int \chi_n d(T^\prime)^j \mu =0.
\end{equation*}
Thus, $\chi_{j,n} \in A^{\perp \perp} = A^{\prime \prime}$.
It follows that the function $h_n$ defined on $\partial A$ as follows
\begin{equation}\label{eq19}
 h_n(t)= \left\{
    \begin{array}{ll}
      f_n(t), & \hbox{if $t \in \{\varphi^j(s_n), j = 0,1, \ldots , m-1\}$;} \\
      0, & \hbox{otherwise}
    \end{array}
  \right.
\end{equation}
is an element of $A^{\prime \prime}$. Next we define $g_n \in A^{\prime \prime}$ as follows
\begin{equation}\label{eq20}
  g_n = h_n + (T^{\prime \prime}h_n - h_n) \Bigg{(} 1 - \frac{1}{\sqrt{n}} \Bigg{)} + \ldots +
((T^{\prime \prime})^{n+1}h_n - (T^{\prime \prime})^nh_n) \Bigg{(} 1 - \frac{1}{\sqrt{n}} \Bigg{)^n}
\end{equation}
Simple calculations similar to~(\ref{eq14}) show that $\|T^{\prime \prime}h_n - h_n\|/\|h_n\| \rightarrow 0$. Hence, $1 \in \sigma_{a.p.}(T,A)$.
It remains to prove that $1 \in \sigma_{usf}(T,A)$. The constructions used in parts (a) and (b) of the proof combined with the condition that $\partial A$ has no isolated points show that we can find $g_n \in A^{\prime \prime}$ such that $\|g_n\|=1$, $T^{\prime \prime}g_n - g_n \rightarrow 0$ and $|g_i|\wedge |g_j| =0$ if $i \neq j$. Thus, the sequence $g_n$ is singular and $1 \in \sigma_{usf}(T^{\prime \prime}, A ^{\prime \prime})= \sigma_{usf}(T,A)$. 
\end{proof}

\begin{remark} \label{r1}
  In the case when $\varphi$ is a homeomorphism of $\partial A$ onto itself Theorem~\ref{t1} is a simple corollary of the results in~\cite{Ki}. Nevertheless, even in this case, as we will see in the next section, Theorem~\ref{t1} provides some useful information about the essential spectra of weighted composition operators on uniform algebras.
\end{remark}

\begin{corollary} \label{c1}
  Assume the conditions of Theorem~\ref{t1}. Assume also that the set of all $\varphi$-periodic points is of first category in $\partial A$. Then the set $\sigma_{usf}(T)$ is rotation invariant.
\end{corollary}
\begin{proof}
  By Theorem 3.7 in~\cite{Ki} the set $\sigma_{a.p.}(T, C(\partial A))$ is rotation invariant.
\end{proof}
Recall the following definition.

\begin{definition} \label{d1}
  A unital uniform algebra $A$ is called \textbf{analytic} if for any open subset $V$ of $\partial A$ and any $f \in A$ the following implication holds $f|V \equiv 0 \Rightarrow f=0$.
\end{definition}

\begin{corollary} \label{c2}
  Assume the conditions of Theorem~\ref{t1}. Assume additionally that the uniform algebra $A$ is analytic and that $T_\varphi^n \neq I, n \in \mathds{N}$. Then the set $\sigma_{usf}(T)$ is rotation invariant.
\end{corollary}

In the next theorem we will investigate the conditions that guarantee that the set $\{|\lambda| : \lambda \in \sigma(T)\}$ is connected.

\begin{theorem} \label{t2}
 Let $A$ be a unital uniform algebra, $T_\varphi$ is an isometric endomorphism of $A$ and $w \in A$. Let $T = wT_\varphi$. Assume that $\varphi$ is an open map of $\partial A$ onto itself. The following conditions are equivalent.
 
 \noindent (1). There is a positive real number $r$ such that $\sigma(T) = \sigma_1 \cup \sigma_2$ where $\sigma_i \neq \emptyset, i = 1,2$, $\sigma_1 \subset \{\lambda \in \mathds{C}: |\lambda| < r\}$, and  $\sigma_2 \subset \{\lambda \in \mathds{C}: |\lambda| > r\}$. 
 
 \noindent (2). There is a closed ideal $J$ of $A$ such that
  \begin{enumerate} [(a)]
    \item $TJ \subset J$ and $\sigma(T,J) \subset \{\lambda \in \mathds{C}: |\lambda| < r\}$. 
    \item The algebra $A|h(J)$, where $h(J) \subset \mathfrak{M}_A$, is the hull of the ideal $J$, is closed in $C(h(J))$.
    \item The algebras $A|h(J)$ and $A|E$, where $E = h(J) \cap \partial A$, are isometric and $A|E$ is closed in $C(E)$.
    \item Moreover, $Int_{\partial A}E \neq \emptyset$.
    \item The operator $T = wT_\varphi$ acts on $A|E$ and $\sigma(T, A|E) \subset \{\lambda \in \mathds{C}: |\lambda| > r\}$.
  \end{enumerate}
   
\end{theorem}

\begin{proof}
 $ (1) \Rightarrow (2)$.  Let $L_1$ and $L_2$ be the spectral subspaces of $A$ corresponding to $\sigma_1$ and $\sigma_2$, respectively. The inequalities $\|T^n(fg)\| \leq  \|T^nf\|\|g\|, f,g \in A, n \in \mathds{N}$ show that $L_1$ is a closed ideal in $A$. We denote this ideal by $J$. The spectral subspace $L_2$ is isomorphic to the factor-algebra $A/L_1$. It is well known that the space of maximal ideals of $A/L_1$ can be identified with the hull $h(J)$ of the ideal $J$.
 
 We claim that the restriction of $\varphi$ on $h(J)$ is a homeomorphism of $h(J)$ onto itself.

We will denote the image of $f \in A$ under the canonical map $j$ of $A$ onto $A/J$ by $\dot{f}$. We claim that $\dot{w}$ is invertible in $A/J$. Assume to the contrary that there is an $m \in \mathfrak{M}_{A/J}$ such that
$\dot{w}(m)= 0$. Let $\sigma$ be the canonical homeomorphism of $\mathfrak{M}_{A/J}$ onto $h(J)$ and let $j = \sigma(m)$. Then $w(j) = 0$. By the definition of $h(J)$ we have $f(j)=0$ for any $f \in L_1$. Let $u \in L_2$. Then $u = Tv, v \in L_2$ and therefore $u(j)=w(j)v(\varphi(j))=0$. Thus, $t(j)= 0$ for any $t \in A$, a contradiction.

Let $(\dot{w})^{-1} = \dot{u}$. We claim that the operator $U=\dot{u}\dot{T}$ is an automorphism of the algebra $A/J$. Let $u \in A$ be such that $ju = \dot{u}$. Then $uw = 1 + z$ where $z \in J$. Let $f, g \in A$, then
$uTfg = uwT_\varphi fg= T_\varphi fg +zT_\varphi fg = (T_\varphi f)(T_\varphi g) +zT_\varphi fg =
(uTf - zT_\varphi f)(uTg - zT_\varphi g) +zT_\varphi fg = (uTf)(uTg) + h$, where $h \in J$. Therefore,
$\dot{u}\dot{T}\dot{f}\dot{g}= (\dot{u}\dot{T}\dot{f})(\dot{u}\dot{T}\dot{g})$. The operator $\dot{u}\dot{T}$ is invertible on $A/J$ and therefore, it is an automorphism of this algebra. Let $\psi$ be the corresponding homeomorphism of $\mathfrak{M}_{A/J}$ onto itself. We claim that $\sigma \circ \psi = \varphi \circ \sigma$. Indeed, we have just seen that for any $f \in A$ we have $\dot{(T_\varphi f)} = \dot{u}\dot{T}\dot{f}$.
Therefore, $\varphi (h) = \sigma(\psi(\sigma^{-1}(h)))$ and the restriction of $\varphi$ onto $h(J)$ is a homeomorphism of $h(J)$ onto itself.
 
 Let $P_1$ and $P_2$ be the spectral projections on $L_1$ and $L_2$, respectively. Let $h \in h(J)$. Then $P_1^\prime \delta_h =0$. Therefore, there are an $n \in \mathds{N}$ and $s_h>r$ such that $|w_n(h)| = \|(T^n)^\prime \delta_h\| = \|(T^n)^\prime P_2^\prime \delta_h \| \geq s_h^n$. It follows from the compactness of $h(J)$ that there are an $p \in \mathds{N}$ and $s >r$ such that
\begin{equation}\label{eq8}
  |w_n(h)| \geq s^n, h \in h(J).
\end{equation}
Let $E = h(J) \cap \partial A$. We claim that there is an $m \in \mathds{N}$ such that 
\begin{equation}\label{eq9}
  E \subset Int(\varphi^{(-m)}(E)).
\end{equation}
Fix an $n \in \mathds{N}$ such that $\|T^n f\| \leq r^n \|f\|, f \in L_1$. If~(\ref{eq9}) does not hold, then it follows from~(\ref{eq8}) that there are an open subset $V$ of $\partial A$ and $s_1, r < s_1 < s$ such that $|w_n(g)| > s_1^n, g \in V$ and $\varphi^n(V) \cap E = \emptyset$. Because the set $\varphi^n(V)$ is open in $\partial A$ and $\partial(J) = \partial A \setminus E$ (see e.g.~\cite[p. 186]{Ka}) there is an $f \in J$ such that $\|f\| = f(g_0) = 1$, where $g_0 \in \varphi^n(V)$. Let $h_0 \in V$ be such that $\varphi^n(h_0) = g_0$. Then, $\|T^n f\| \geq |(T^n f)(h_0)|= |w_n(h_0)f(\varphi^n(h_0))| = |w_n(h_0)| > s_1^n$, a contradiction.

We claim that $A|E$ and $A|h(J)$ are closed subalgebras of $C(E)$ and $C(h(J))$, respectively. We need to prove that there is a $C>0$ such that $\|f\| \leq C\|f|E\|, f \in L_2$. Equivalently, we have to prove that there are an $n \in \mathds{N}$ and $D>0$   
 such that 
  \begin{equation}\label{eq10}
   \|T^nf\| \leq D\|T^nf|E\|, f \in L_2 .
 \end{equation}
 On the previous step we have proved that there is an $n \in \mathds{N}$ such that
  \begin{equation}\label{eq11}
   \min_{h \in E} |w_n(h)| > \sup_{ g \in \partial A \setminus \varphi^{(-n)}(E)} |w_n(g)|.
 \end{equation}
 It is immediate to see that~(\ref{eq10}) follows from~(\ref{eq11}). 
 
 Thus, the implication $(1) \Rightarrow (2)$ is proved.
 
 \noindent $(2) \Rightarrow (1)$. This implication is almost trivial. Indeed, it follows from (2) that $\sigma(T, J) \subset \{\lambda \in \mathds{C} : |\lambda| < r\}$ and $\sigma(\dot{T}, A/J) \subset \{\lambda \in \mathds{C} : |\lambda| > r\}$. By the three spaces theorem $\sigma(T,A) \cap \{\lambda \in \mathds{C} : |\lambda| = r\} = \emptyset$
 \end{proof}

 \begin{corollary} \label{c3}
   Assume conditions of Theorem~\ref{t2}. Assume additionally that the algebra $A$ is analytic. Then the set $\{|\lambda|: \lambda \in \sigma(T,A)\}$ is connected. If we also assume that $T_\varphi^n \neq I, n \in \mathds{N}$, then $\sigma(T,A)$ is a connected rotation invariant subset of $\mathds{C}$.
 \end{corollary}
 
 The next proposition will be useful for us in the sequel.

 \begin{proposition} \label{p1} Let $A$ be a unital uniform algebra. Assume that $\partial A$ has no isolated points. Let $T_\varphi$ be a periodic automorphism of $A$, i.e. there is an $n \in \mathds{N}$ such that $T_\varphi^n = I$. Let $w \in A$ and $T = wT_\varphi$. For any $t \in \mathfrak{M}_A$ let $p(t)$ be the smallest positive integer such that $\varphi^{p(t)}(t) = t$. Then
\begin{enumerate}
  \item $\sigma(T) = \{\lambda \in \mathds{C}: \exists t \in \mathfrak{M}_A, \lambda^{p(t)} = w_{p(t)}(t)\}$.
  \item $\sigma_{a.p.}(T) = \sigma_{sf}(t) =\{\lambda \in \mathds{C}: \exists t \in \partial A, \lambda^{p(t)} = w_{p(t)}(t)\}$
\end{enumerate}
  \end{proposition}
\begin{proof} (a) We prove the inclusion $\sigma =\{\lambda \in \mathds{C}: \exists t \in \mathfrak{M}_A, \lambda^{p(t)} = w_{p(t)}(t)\} \subseteq \sigma(T)$. The implication $0 \in \sigma \rightarrow 0 \in \sigma(T)$ is obvious. Let $\lambda \in \sigma \setminus \{0\}$ and let $t \in \mathfrak{M}_A$ be such that $\lambda^{p(t)} = w_{p(t)}(t)\}$. Denote by $\delta_s$ the functional $\delta_s(f) = f(s), f \in A, s \in \mathfrak{M}_A$. Let
$F = \delta_t + \lambda^{-1}w(t)\delta_{\varphi(t)} + \ldots + \lambda^{1-p(t)}w_{p(t)-1}(t)\delta_{\varphi^{p(t)-1}(t)}$. Then $T^\prime F = \lambda F$. Moreover, because all the points $t, \varphi(t), \ldots, \varphi^{p(t)-1}(t)$ are distinct, $F \neq 0$ (see e.g.~\cite{Ka}).

(b) To prove that $\sigma(T) \subseteq \sigma$ assume that $\lambda \in \sigma(T)$. Let $n$ be the smallest positive integer such that $T_\varphi^n = I$. Then $\lambda^n \in \sigma(T^n) = \sigma(w_n)$, and therefore there is a 
$t \in \mathfrak{M}_A$ such that $\lambda^n = w_n(t)$. But $p(t)$ divides $n$ and therefore $\lambda^{p(t)} = w_{p(t)}(t)$.

(c) Let $t \in \partial A$ and $\lambda \in \mathds{C}$ be such that $\lambda^{p(t)} = w_{p(t)}(t)$. The case $\lambda =0$ is easy, and therefore we will assume that $\lambda \neq 0$. There are two possibilities.

\noindent (c1) There are pairwise distinct $p$-points $t_n \in \partial A$ such that $p(t_n) = p(t)$ and $w_{p(t)}(t_n) \rightarrow w_{p(t)}(t)$. Let $\lambda_n \in \mathds{C}$ be such that $\lambda_n^{p(t)} = w_{p(t)}(t_n)$ and $\lambda_n \rightarrow \lambda$. Then $\chi_n = \chi_{t_n} \in A^{\prime \prime}$ and $T^{\prime \prime}F_n = \lambda_n G_n$, where $G_n = \chi_n + \lambda^{-1}T^{\prime \prime}\chi_n + \ldots + \lambda^{1-n}(T^{\prime \prime})^{n-1}\chi_n$. Clearly the sequence $G_n$ is singular in $A^{\prime \prime}$, and therefore $\lambda \in \sigma_{usf}(T)$.

Similarly, if $F_n = \delta_{t_n} + \lambda^{-1}w(t)\delta_{\varphi(t_n)} + \ldots + \lambda^{1-p(t)}w_{p(t)-1}(t)\delta_{\varphi^{p(t)-1}(t_n)}$, then $T^\prime F_n = \lambda_n F_n$. It follows easily from the fact that $t_n$ are $p$-points that the sequence $F_n$ is singular in $A^\prime$. Thus, $\lambda \in \sigma_{sf}(T)$.

\noindent (c2)  There are pairwise distinct $p$-points $t_n \in \partial A$ such that  $p(t_n) = p$, $p(t)$ divides $p$, and $w_p(t_n) \rightarrow w_{p(t)}(t)$. We can apply the same reasoning as in the case (c1).
 \end{proof}

From Proposition~\ref{p1} we derive our next result

\begin{theorem} \label{t3}
  Let $A$ be a unital analytic uniform algebra, $U = T_\varphi$ be a non-periodic automorphism of $A$, and $T = wT_\varphi$. Assume that
\begin{enumerate} [(a)]
  \item $\partial A$ is a connected set.
  \item There are periodic automorphisms, $U_n = T_{\varphi_n}$ of $A$ such that for any $t \in \partial A$ we have $\varphi_n(t) \rightarrow \varphi(t)$.
\item Let $\mu \in (C(\partial A))^\prime$ be a $\varphi$-invariant probability measure. Then $\mu$ is $\varphi_n$-invariant, $n \in \mathds{N}$.
\end{enumerate}
Then 
\begin{enumerate}
  \item $\sigma_{a.p.}(T) = \sigma_{sf}(T)$.
  \item The sets $\sigma(T)$ and $\sigma_{sf}(T)$ are connected rotation invariant subsets of $\mathds{C}$.
  \item If either $w \in A^{-1}$ or $\min \limits_{s \in \partial A} |w(s)| = 0$, then $\sigma(T) = \sigma_{sf}(T)$.
\end{enumerate}
\end{theorem}

\begin{proof} The set $\sigma(T)$ is connected and rotation invariant by Corollary~\ref{c3}. The set $\sigma_{a.p.}(T) = \sigma_{usf}(T)$ is rotation invariant by Corollary~\ref{c2}.

To prove that the set $\sigma_{a.p.}(T)$ is connected assume first that $w$ is invertible in $C(\partial A)$.
It follows from~\cite{Ki} that $|\sigma(T)| = [\rho_{min}(T), \rho(T)]$. Let $|\lambda| \in [\rho_{min}(T), \rho(T)]$. For any $n \in \mathds{N}$ let $T_n = wU_n$. It follows from (c) and~\cite{Ki} that
$\rho_{min}(T_n) \leq \rho_{min}(T)$ and $\rho(T_n) \geq \rho(T)$. Next, it follows from (a) and from (2) in Proposition~\ref{p1} that $|\lambda| \in |\sigma_{a.p.}(T_n)|$. Therefore (see~\cite{Ki}), there is $t_n \in \partial A$ such that
  \begin{equation}\label{eq34}
  |w_{m,n}(t_n)| \geq |\lambda|^m, \; \text{and} \; |w_{m,n}(\varphi_n^{-m}(t_n))| \leq |\lambda|^m, m \in \mathds{N}.
\end{equation}
where $w_{m,n}(t) = w(t)w(\varphi_n(t)) \ldots w(\varphi_n^{m-1}(t)$. Let $t \in \partial A$ be a limit point of the sequence $t_n$. It follows from~(\ref{eq34}) and from (b) that
\begin{equation}\label{eq35}
  |w_m(t)| \geq |\lambda|^m, \; \text{and} \; |w_m(\varphi^{-m}(t))| \leq |\lambda|^m, m \in \mathds{N}.
\end{equation}
It follows from~(\ref{eq35}) (see~\cite{Ki}) that $\lambda \in \sigma_{a.p.}(T)$ and by Theorem~\ref{t1} $\lambda \in \sigma_{usf}(T)$.

To prove that $\lambda \in \sigma_{lsf}(T)$ we notice that applying the previous arguments to the operator $S$,
$Sf(z) = w(z)f(\varphi^{-1}(z)), f \in \mathds{A}, z \in \mathds{D}$ we can find a point $s \in \partial A$ such that
\begin{equation}\label{eq36}
  |w_m(s)| \leq |\lambda|^m, \; \text{and} \; |w_m(\varphi^{-m}(s))| \geq |\lambda|^m, m \in \mathds{N}.
\end{equation}
It follows from~(\ref{eq36}) and Proposition~\ref{p2} below that $\lambda \in \sigma_{lsf}(T)$.

To finish the proof it remains to notice that if $\min \limits_{s \in \partial A} |w(s)|=0$ then $|\sigma(T)| = [0, \rho(T)]$ and $|\sigma(T_n)|=[0, \rho(T_n)]$, $n \in \mathds{N}$, and to apply the arguments from the previous part of the proof.
\end{proof}

\begin{remark} \label{r4}
  The proof of Theorem~\ref{t3} shows that instead of condition (a) we can assume a weaker condition

\noindent ($\acute{a}$) $\partial A$ is the union of a finite number of connected sets $\partial_1, \ldots , \partial_m$ such that $\varphi(\partial_j) = \partial_j$ and $\varphi_n(\partial_j) = \partial_j$ , $j= 1, \ldots , m$, $n \in \mathds{N}$.
\end{remark}

The next proposition and Corollary~\ref{c4} will enable us to obtain in some cases the description of the lower semi-Fredholm spectrum.

\begin{proposition} \label{p2}
 Let $A$ be a unital uniform algebra, $T_\varphi$ be an isometric endomorphism of $A$ and $w \in A$. Assume that the map $\varphi : \partial A \rightarrow \partial A$ is open, that $\partial A$ has no isolated points, and that the set of $\varphi$-periodic points is a set of first category in $\partial A$. Assume also that there are $\lambda \in \mathds{C}$ and $t_n \in \partial A, n \in \mathds{Z}$ such that
\begin{equation}\label{eq22}
\begin{split}
& \varphi(t_n) = t_{n+1}, n \in \mathds{Z}, \\
 & |w_n(t_0)| \leq |\lambda|^n, \; \text{and} \; |w_n(t_{-n})| \geq |\lambda|^n, n \in \mathds{N}.
\end{split}
\end{equation}
Then, $\lambda \in \sigma_{lsf}(T)$.
\end{proposition}

\begin{proof}
  Let $\lambda = 0$. It follows from~(\ref{eq22}) and our assumptions that there are pairwise distinct $p$-points $s_n \in \partial A$ such that $|w(s_n)| < 1/n$. Then $\|\delta_{s_n}\| =1$ and $T^\prime \delta_{s_n} = w(s_n)\delta_{\varphi(s_n)} \rightarrow 0$. The sequence $\delta_{s_n}$ is obviously singular and therefore 
$0 \in \sigma_{lsf}(T)$.

If $\lambda \neq 0$ we can without loss of generality assume that $\lambda = 1$ and that
\begin{equation}\label{eq23}
\begin{split}
& \varphi(t_n) = t_{n+1}, n \in \mathds{Z}, \\
 & |w_n(t_0)| \leq 1, \; \text{and} \; |w_n(t_{-n})| \geq 1, n \in \mathds{N}.
\end{split}
\end{equation}
From our assumptions follows that for any $n \in \mathds{N}$ there are pairwise distinct points $u_j \in \partial A, -n \leq j \leq n+1$ such that $u_{n+1}$ is a $p$-point and
\begin{equation}\label{eq24}
\begin{split}
& \varphi(u_j) = u_{j+1}, -n \leq j \leq n, \\
 & |w_j(u_0)| \leq 2, j=1, \ldots , n+1 \; \text{and} \; |w_j(u_{-j})| \geq 1/2, \; j = 1, \ldots , n.
\end{split}
\end{equation}
We define $\mu_n \in A^\prime$ as follows
\begin{equation}\label{eq25}
 \mu_n = \frac{1}{w_n(u_{-n})} \sum \limits_{j=-n}^n \bigg{(} 1 -\frac{1}{\sqrt{n}} \bigg{)}^{|j|} (T^\prime)^{j+n} \delta_{u_{-n}}.
\end{equation}

In particular, $\|\mu_n\| \geq 1$.
We claim that $\|T^\prime \mu_n - \mu_n\|/\|\mu_n\| \mathop \rightarrow \limits_{n \rightarrow \infty} 0$. First let us notice that
\begin{equation}\label{eq27}
  \|\mu_n\| = \frac{1}{|w_n(u_{-n})|} \sum \limits_{j=-n}^n \bigg{(} 1 -\frac{1}{\sqrt{n}} \bigg{)}^{|j|} |w_{j+n}(u_{-n})|.
\end{equation}
Indeed, because $u_{-n}$ is a $p$-point, there is an $f \in A$ such that $f(u_{-n})=1 = \|f\|$ and $|f(\varphi^j(u_{-n})| < 1, j = 1, \ldots , 2n$. Let $g = \sum \limits_{j=0}^{2n} T_\varphi^j f$.
Then, $\|g^m\| \mathop \rightarrow \limits_{m \rightarrow \infty} 1$ and
\begin{equation*}
  \lim \limits_{m \rightarrow \infty} \int g^m d\mu_n = \frac{1}{|w_n(u_{-n})|} \sum \limits_{j=-n}^n \bigg{(} 1 -\frac{1}{\sqrt{n}} \bigg{)}^{|j|} |w_{j+n}(u_{-n})|.
\end{equation*}

Similarly to~(\ref{eq14}) we can write
 \begin{equation}\label{eq26}
\begin{split}
 & T^{\prime}\mu_n - \mu_n = \frac{1}{w_n(u_{-n})}\bigg{[} -\bigg{(}(1-\frac{1}{\sqrt{n}} \bigg{)}^n \delta_{u_{-n}} \\
& - \frac{1}{\sqrt{n}} \sum \limits_{i=1}^n \bigg{(} 1 -\frac{1}{\sqrt{n}} \bigg{)}^{|i-n|} (T^{\prime})^i \delta_{u_{-n}} \\
& + \frac{1}{\sqrt{n}} \sum \limits_{i=n+1}^{2n} \bigg{(} 1 -\frac{1}{\sqrt{n}} \bigg{)}^{|i-n|} (T^{\prime})^i \delta_{u_{-n}} \\
& + \bigg{(}(1-\frac{1}{\sqrt{n}} \bigg{)}^n (T^{\prime})^{2n+1} \delta_{u_{-n}}\Bigg{]}.
 \end{split}
\end{equation}
It follows from~(\ref{eq26}), ~(\ref{eq24}), and~(\ref{eq27}) that
\begin{equation}\label{eq28}
\begin{split}
&  \|T^\prime \mu_n - \mu_n\| \leq \frac{1}{\sqrt{n}} \|\mu_n\| \\
& +  \frac{1}{|w_n(u_{-n})|}\bigg{(}(1-\frac{1}{\sqrt{n}} \bigg{)}^n +|w_{n+1}(u_0)|\bigg{(}(1-\frac{1}{\sqrt{n}} \bigg{)}^n \\  
& \leq \frac{1}{\sqrt{n}} \|\mu_n\| + 4\bigg{(}(1-\frac{1}{\sqrt{n}} \bigg{)}^n.
\end{split}
\end{equation}
Combining~(\ref{eq28}) and the inequality $\|\mu_n\| \geq 1$ we obtain that $\|T^\prime \mu_n - \mu_n\|/\|\mu_n\| \mathop \rightarrow \limits_{n \rightarrow \infty} 0$. Thus, $1 \in \sigma_{a.p.}(T^\prime)$.
Because $\partial A$ has no isolated points we can construct the sequence $\mu_n$ in such a way that $|\mu_m| \wedge |\mu_n| =0, m \neq n$. Therefore the sequence $\mu_n$ is singular and $1 \in \sigma_{lsw}(T)$.
\end{proof}

\begin{remark} \label{r2}
  If in the statement of Proposition~\ref{p2} we drop the requirement that the set of all $\varphi$-periodic points is of first category in $\partial A$, then the reasoning similar to the proof of this proposition shows that $\sigma_{lsf}(T) \cap \lambda \mathds{T} \neq \emptyset$.
\end{remark}

The next corollary follows directly from Proposition~\ref{p2} and the proofs of Lemmas 4.3 and 4.4 in~\cite{KO1}

\begin{corollary} \label{c4}
 Let $A$ be a unital uniform algebra, $T_\varphi$ be an isometric endomorphism of $A$ and $w \in A$. Assume that the map $\varphi : \partial A \rightarrow \partial A$ is open, that $\partial A$ has no isolated points, and that the set of $\varphi$-periodic points is a set of first category in $\partial A$. If there is $\lambda \in \sigma(T)$ such that $|\lambda| > \rho_{min}(T)$ and $\lambda \not \in \sigma_{lsf}(T)$, then there is a proper closed subset $E \subset \partial A$ such that $\varphi(E) = \varphi^{(-1)}(E) =E$. 
\end{corollary}

 \bigskip

 \section{Essential spectra of weighted automorphisms of disc-algebra}

The spectra of weighted automorphisms of the disk-algebra were first described by the late Herbert Kamowitz in~\cite{Kam}. Here we will complement Kamowitz's results by providing a  partial (see Proposition~\ref{p8} and Problem~\ref{pr1})  description of the essential spectra of these operators.

In accordance with the well known classification of Möbius transformations we have to consider the following cases.

\noindent I. The map $\varphi$ is an elliptic Möbius transformation and there is an $n \in \mathds{N}$ such that $T_\varphi^n = I$.
In other words, the map $\varphi$ is topologically conjugate to a map $\psi$, $\psi(z) = \exp{(ir\pi)}z$, where $r$ is a rational number.The following proposition is special case of Proposition~\ref{p1}.

\begin{proposition} \label{p3}
  Let $\mathds{A}$ be the disc-algebra and $\varphi$ be an elliptic Möbius transformation topologically conjugate to a rational rotation. Let $m$ be the smallest positive integer such that $\varphi^m(z) \equiv z, z \in \mathds{D}$. Let $w \in \mathds{A}$ and $T = wT_\varphi$. Then:
  \begin{enumerate}
    \item $\sigma(T) = \sigma_w(T) = \{\lambda \in \mathds{C} : \lambda^m \in w_m(\mathds{D}) \}$.
    \item $\sigma_{a.p.}(T) = \sigma_f(T) = \{\lambda \in \mathds{C} : \lambda^m \in w_m(\mathds{T}) \}$.
      \end{enumerate}
\end{proposition}

\noindent II. The map $\varphi$ is an elliptic Möbius transformation topologically conjugate to a rotation $\psi$, $\psi(z) = \exp{(i\alpha\pi)}z$, where $\alpha$ is an irrational real number. We denote by $z_0$ the fixed point of $\varphi$ in $\mathds{D}$.

\begin{proposition} \label{p4}
Let $\mathds{A}$ be the disc-algebra and $\varphi$ be an elliptic Möbius transformation topologically conjugate to an irrational rotation. Let $z_0 = r_0 \exp{i\theta_0}$ be the fixed point of $\varphi$ in $\mathds{U}$. Let $w \in \mathds{A}$, $w \not \equiv 0$, and $T = wT_\varphi$. Assume that $w(z) = (z-z_0)^n w_1(z)$, where $n \geq 0$ and $w_1(z_0) \neq 0$. Then:
\begin{equation} \label{eq29}
\begin{split}
 &  \rho(T) = \rho_{min}(T) = \exp{ \int \limits_0^{2\pi} \ln{|w(e^{i\theta})|}\frac{1-r_0^2}{1+r_0^2-2r_0\cos{\theta_0}} \frac{1}{2\pi} d\theta} = \\
& =|w_1(z_0)|,
\end{split}
 \end{equation}
  and
\begin{enumerate}
  \item If $w$ is an invertible element of $\mathds{A}$, then $\sigma_{sf}(T) = \sigma(T) = \rho(T)\mathds{T}$.
   \item If $w \not \in \mathds{A}^{-1}$ but $|w| > 0$ on $\mathds{T}$, then $\sigma(T) = \sigma_w(T) = \rho(T)\mathds{D}$ and 
   $\sigma_f(T) = \rho(T)\mathds{T}$.
 \item If $\min \limits_{\mathds{T}} |w| = 0$, then $\sigma(T) = \sigma_{sf}(T) = \rho(T)\mathds{D}$.
\end{enumerate}
 \end{proposition}

\begin{proof} The formula~(\ref{eq29}) follows from~\cite[Theorem 3.23]{Ki}.

(1) is trivial because it follows from~(\ref{eq29}) that $\rho(T^{-1}) =1/\rho(T)$. 

(2) follows immediately from~\cite[Theorem 4.2]{Ki} and the obvious fact that for any $\lambda \in \rho(T)\mathds{D}$  $def(\lambda I - T) =\Sigma$, where $\Sigma$ is the sum of multiplicities of all zeros of $w$ in $\mathds{D}$.

(3) Follows directly from Theorem~\ref{t3}.
\end{proof}
\noindent III. The map $\varphi$ is a parabolic Möbius transformation.

\begin{proposition} \label{p5}
 Let $\mathds{A}$ be the disc-algebra and $\varphi$ be a parabolic Möbius transformation. Let $\zeta \in \mathds{T}$ be the fixed point of $\varphi$. Let $w \in \mathds{A}$, $w \not \equiv 0$, and $T = wT_\varphi$. Then:
 $\rho(T) = \rho_{min}(T) = |w(\zeta)|$.
 \begin{enumerate}
   \item If $w$ is an invertible element of $\mathds{A}$, then $\sigma_{sf}(T) = \sigma(T) = w(\zeta)\mathds{T}$.
   \item If $w \not \in \mathds{A}^{-1}$ but $|w| > 0$ on $\mathds{T}$, then $\sigma(T) = \sigma_w(T) = w(\zeta)\mathds{D}$ and 
   $\sigma_f(T) = w(\zeta)\mathds{T}$.
 \item If $w(\zeta) =0$ then $\sigma_{sf}(T) = \sigma(T) = \{0\}$.
   \item If $\min \limits_{\mathds{T}} |w| = 0$ but $|w(\zeta)| >0$, then $\sigma(T) = \sigma_{sf}(T) = w(\zeta)\mathds{D}$.
 \end{enumerate}
\end{proposition}

\begin{proof} The statements (1) - (3) are trivial. To prove (4) assume that $0 < |\lambda|<|w(\zeta)|$. Let $t \in \mathds{T}$ be such that $w(\varphi^{-1}(t)) = 0$ and $w(\varphi^n(t)) \neq 0, n = 0, 1, \ldots$. Let 
$f = \sum \limits_{n=0}^\infty \frac{\lambda^n}{w_n(t)}\chi_{\varphi^n(t)}$. By Lemma~\ref{l1} $f \in \mathds{A}^{\prime \prime}$ (recall that every point of $\mathds{T}$ is a peak point of $\mathds{A}$). Obviously $T^{\prime \prime}f = \lambda f$. Therefore $\lambda \in \sigma_{a.p.}(T)$ and by Theorem~\ref{t1} 
$\lambda \in \sigma_{usf}(T)$.

Let $s \in \mathds{T}$ be such that $w(\varphi(s)) = 0$ and $w(\varphi^{-n}(s)) \neq 0, n = 0, 1, \ldots$.
 Then at the point $s$ the inequalities~(\ref{eq22}) are satisfied, and by Proposition~\ref{p2} $\lambda \in \sigma_{lsf}(T)$. 
\end{proof}

\noindent IV. The map $\varphi$ is a hyperbolic Möbius transformation. Let $\zeta_1, \zeta_2 \in \mathds{T}$ be the fixed points of $\varphi$. Without loss of generality we can assume that $|\varphi^\prime(\zeta_1)| < 1$ and $|\varphi^\prime(\zeta_2)| > 1$. The 
description of the spectrum and the essential spectra of $T = wT_\varphi$ depends on the relation between the numbers $|w(\zeta_1)|$ and $|w(\zeta_2)|$. This description will be provided in the next three propositions. 

\begin{proposition} \label{p6}
 Let $\mathds{A}$ be the disc-algebra and $\varphi$ be a hyperbolic Möbius transformation. Let $w \in \mathds{A}$, $w \not \equiv 0$, and $T = wT_\varphi$. Assume that $|w(\zeta_1)| = |w(\zeta_2)| = \rho$. Then:
 $\rho(T) = \rho_{min}(T) = \rho$.
 \begin{enumerate}
   \item If $w$ is an invertible element of $\mathds{A}$, then $\sigma_{sf}(T) = \sigma(T) = \rho\mathds{T}$.
   \item If $w \not \in \mathds{A}^{-1}$ but $|w| > 0$ on $\mathds{T}$, then $\sigma(T) = \sigma_w(T) = \rho\mathds{D}$ and 
   $\sigma_f(T) = \rho\mathds{T}$.
 \item If $\rho =0$ then $\sigma_{sf}(T) = \sigma(T) = \{0\}$.
   \item If $\min \limits_{\mathds{T}} |w| = 0$ but $\rho >0$, then $\sigma(T) = \sigma_{sf}(T) = \rho\mathds{D}$.
 \end{enumerate}
\end{proposition}
\begin{proof}
  The proof is identical to the proof of Proposition~\ref{p5}.
\end{proof}

\begin{proposition} \label{p7}
  Let $\mathds{A}$ be the disc-algebra and $\varphi$ be a hyperbolic Möbius transformation. Let $w \in \mathds{A}$, $w \not \equiv 0$, and $T = wT_\varphi$. Assume that $|w(\zeta_1)| < |w(\zeta_2)|$. Then:
 $\rho(T) = |w(\zeta_2)|$ and $\rho_{min}(T) =|w(\zeta_1)|$.
 \begin{enumerate}
   \item If $w$ is an invertible element of $\mathds{A}$, then

 $ \sigma_{lsf}(T) = \sigma_f(T) = \sigma(T) = \{\lambda \in \mathds{C}: |w(\zeta_1)| \leq |\lambda| \leq |w(\zeta_2)|\}$

  and $\sigma_{sf}(T) = |w(\zeta_1)|\mathds{T} \cup |w(\zeta_2)|\mathds{T}$.
   \item If $w \not \in \mathds{A}^{-1}$ but $|w| > 0$ on $\mathds{T} \setminus \{\zeta_1\}$, then

 $\sigma_w(T) = \sigma(T) = |w(\zeta_2)|\mathds{D}$, 

$ \sigma_{lsf}(T) = \sigma_f(T) = \{\lambda \in \mathds{C}: |w(\zeta_1)| \leq |\lambda| \leq |w(\zeta_2)|\}$,
 and 

       $ \sigma_{sf}(T) = \sigma_{usf}(T) = |w(\zeta_1)|\mathds{T} \cup |w(\zeta_2)|\mathds{T}$.   
   \item If there is an $s \in \mathds{T}$ such that $s \neq \zeta_1$ and $w(s) = 0$, then 

$\sigma_f(T) = \sigma(T) = |w(\zeta_2)|\mathds{D}$, 

$ \sigma_{lsf}(T)  = \{\lambda \in \mathds{C}: |w(\zeta_1)| \leq |\lambda| \leq |w(\zeta_2)|\}$,
 and

 $\sigma_{sf}(T) = \sigma_{usf}(T) = |w(\zeta_1)|\mathds{D} \cup |w(\zeta_2)|\mathds{T}$.
    \end{enumerate}
\end{proposition}

\begin{proof} To prove (1) and (2) it is sufficient to notice  that
\begin{enumerate} [(a)]
  \item By~\cite[Theorem 3.29]{Ki}  $\{\lambda \in \mathds{C}: |w(\zeta_1)| < |\lambda < |w(\zeta_2)|\} = \sigma_r(T)$
  \item If $s \in \mathds{T} \setminus \{\zeta_1, \zeta_2\}$, then at the point $s$ the inequalities~(\ref{eq22}) are satisfied and therefore we can apply Proposition~\ref{p2}.
\end{enumerate}
 To prove (3) we apply the same reasoning as in the proof of Proposition~\ref{p5}. 
\end{proof}

\begin{proposition} \label{p8}
  Let $\mathds{A}$ be the disc-algebra and $\varphi$ be a hyperbolic Möbius transformation. Let $w \in \mathds{A}$, $w \not \equiv 0$, and $T = wT_\varphi$. Assume that $|w(\zeta_1)| > |w(\zeta_2)|$. Then:
 $\rho(T) = |w(\zeta_2)|$ and $\rho_{min}(T) =|w(\zeta_1)|$. 
\begin{enumerate}
  \item If $w$ is an invertible element of $\mathds{A}$, then
\begin{equation*}
\begin{split}
 & \sigma_{usf}(T) = \sigma_{a.p.}(T) = \sigma(T) = \{\lambda \in \mathds{C}: |w(\zeta_2) \leq |\lambda| <|w(\zeta_1)|\}  \\
& \text{and} \; w(\zeta_1)\mathds{T} \cup w(\zeta_2)\mathds{T} \subseteq \sigma_{lsf}(T).
\end{split}
\end{equation*}
  \item If $w \not \in \mathds{A}^{-1}$ but $|w| > 0$ on $\mathds{T} \setminus \{\zeta_2\}$, then
\begin{equation*}
\begin{split}
& \sigma(T) = \sigma_w(T) = w(\zeta_1)\mathds{D} \\
 & \sigma_{usf}(T) = \sigma_{a.p.}(T) =  \{\lambda \in \mathds{C}: |w(\zeta_2) \leq |\lambda| <|w(\zeta_1)|\}  \\
& \text{and} \; w(\zeta_1)\mathds{T} \cup w(\zeta_2)\mathds{T} \subseteq \sigma_{lsf}(T).
\end{split}
\end{equation*}
  \item If there is an $s \in \mathds{T}$ such that $s \neq \zeta_1$ and $w(s) = 0$, then 
\begin{equation*}
\begin{split}
& \sigma(T) = \sigma_f(T) = w(\zeta_1)\mathds{D} \\
 & \sigma_{usf}(T) = \sigma_{a.p.}(T) =  \{\lambda \in \mathds{C}: |w(\zeta_2) \leq |\lambda| <|w(\zeta_1)|\}  \\
& \text{and} \; w(\zeta_1)\mathds{T} \cup w(\zeta_2)\mathds{T} \subseteq \sigma_{lsf}(T).
\end{split}
\end{equation*}
\end{enumerate}
\end{proposition}

\begin{proof}
  The proof follows directly from~\cite[Theorem 3.29]{Ki}  and Theorems~\ref{t1} and~\ref{t2}. 
\end{proof}

\begin{problem} \label{pr1}
  Describe the set $\sigma_{lsf}(T)$ assuming the conditions of Proposition~\ref{p8}.
\end{problem}

\section{Examples}

\begin{example} \label{e3}
  In this example we consider the unit polydisc
\begin{equation*}
  \mathds{D}^n = \{(z_1, \ldots , z_n) \in \mathds{C}^n : |z_i| \leq 1, i=1, \dots , n \}
\end{equation*}
The algebra $\mathds{A}^n$ is the polydisc algebra - the algebra of all functions analytic in
\begin{equation*}
  \mathds{U}^n = \{(z_1, \ldots , z_n) \in \mathds{C}^n : |z_i| < 1, i=1, \dots , n \}.
\end{equation*}
and continuous in $\mathds{D}^n$. It is well known (see e.g.~\cite{Ru}) that the space of maximal ideals of the algebra $\mathds{A}^n$ coincides with $\mathds{D}^n$ and its Shilov boundary coincides with the torus $\mathds{T}^n$.

Let $\varphi(z_1, \ldots , z_n) = (\alpha_1 z_1, \ldots, \alpha_n z_n)$, where $\alpha_i = exp(i\gamma_i)$ and the real numbers $\gamma_i, i = 1, \ldots, n$, are linearly independent over the field of rational real numbers.
Let $w \in \mathds{A}^n$ and $T = wT_\varphi$. It follows from our results that
\begin{enumerate}
  \item If $w$ is invertible in $\mathds{A}^n$ then $\sigma(T) = \sigma_f(T) = |w(0, \ldots, 0)|\mathds{T}$.
  \item If $w$ is not invertible in $\mathds{A}^n$, but invertible in $C(\mathds{T}^n)$ then $\sigma(T) = \sigma_{usf}(T) = \rho(T)\mathds{D}$, where
\begin{equation}\label{eq37}
  \rho(T) = exp \int \limits_{\mathds{T}^n} ln|w| dm_n,
\end{equation}
and $m_n$ is the standard normalized Lebesgue measure on $\mathds{T}^n$.
  \item If $w$ is not invertible in $C(\mathds{T}^n)$, then $\sigma(T)= \sigma_{sf}(T) = \rho(T)\mathds{D}$, where $\rho(T)$ is given by~(\ref{eq37}).
\end{enumerate}
\end{example}

\begin{remark} \label{r5}
  The authors plan to provide a detailed description of the spectra of weighted automorphisms of $\mathds{A}^n$ in a separate publication.
\end{remark}

\begin{example} \label{e2}. 
  \textbf{Weighted endomorphisms of the disk-algebra}

Any isometric endomorphism of the disk-algebra $A$ is of the form
\begin{equation*}
 (T_\varphi f)(z) = f(B(z)), f \in A, z \in \mathds{D},
\end{equation*}
where $B$ is a finite Blaschke product. We assume that $B$ has at least two factors, i.e., it is not a Möbius transformation.

Let $w \in A$ and $T = wT_\varphi$. Theorems~\ref{t1} and~\ref{t2} guarantee that $\sigma(T)$ is a disk (or the singleton $\{0\}$) and that the sets $\sigma_{a.p.} (T) = \sigma_{usf}(T)$ are rotation invariant. Moreover,
it follows from Corollary~\ref{c4} that $\{\lambda \in \mathds{C}: \rho_{min}(T) \leq |\lambda| \leq \rho(T)\} \subseteq \sigma_{lsf}(T)$.

The complicated dynamics of finite Blaschke products on the unit circle currently prevents us from obtaining more information about the spectrum and the essential spectra of $T$.

Even the computation of the spectral radius $\rho(T)$ becomes nontrivial (unless $B$ is a Möbius transformation). The well-known formula (see e.g.~\cite{Ki})
\begin{equation*}
   \rho(T) = \max \limits_{\mu \in M_\varphi} \exp \int \ln{|w|} d\mu,
\end{equation*}
where $M_B$ is the set of all $B$-invariant probability measures in $C(\mathds{T})^\prime$, becomes practically useless, because the set $M_B$ is enormous and does not allow a simple description (see~\cite{Fu}).

Let us consider for example the following two seemingly simple operators.
\begin{equation*}
\begin{split}
 & (T_1f)(z) =\Big{(} \frac{1+z}{2} \Big{)}f(z^2),  \\
 & (T_2f)(z) =\Big{(} \frac{1-z}{2} \Big{)}f(z^2), \\
 & f \in A, z \in \mathds{D}.
 \end{split}
\end{equation*}
It is trivial that $\rho(T_1) = 1$, and some elementary calculations show that $\rho(T_2) =\frac{1}{2}$.
Thus, $\sigma(T_1) = \mathds{D}$ and $\sigma(T_2) = \frac{1}{2}\mathds{D}$. But, beyond the fact that the sets $\sigma_{a.p.}(T_i), i = 1,2$, are rotation invariant and that the trivial inclusions  $\{0\} \cup \mathds{T} \subset \sigma_{a.p.}(T_1)$  and $\{0\} \cup \frac{1}{2}\mathds{T} \subset \sigma_{a.p.}(T_2)$ hold,  we cannot say anything else about the sets $\sigma_{a.p.}(T_i), i = 1,2$.
\end{example}

\end{document}